\documentclass[11pt]{amsart}
\usepackage[utf8]{inputenc}
\usepackage[T1]{fontenc}
\usepackage{amssymb}
\usepackage[english]{babel}
\usepackage{enumitem}

\theoremstyle{plain}
\newtheorem{theorem}{Theorem}[section]
\newtheorem{lemma}[theorem]{Lemma}
\newtheorem{proposition}[theorem]{Proposition}

\theoremstyle{definition}

\newtheorem{notation}[theorem]{Notation}

\theoremstyle{remark}
\newtheorem{remark}[theorem]{Remark}

\newcommand{\Aut}{\operatorname{Aut}}
\newcommand{\Inn}{\operatorname{Inn}}
\newcommand{\Out}{\operatorname{Out}}
\newcommand{\Outc}{\Out_c}
\newcommand{\OutCol}{\Out_{\mathrm{Col}}}
\newcommand{\Autc}{\Aut_c}
\newcommand{\AutCol}{\Aut_{\mathrm{Col}}}
\newcommand{\AutZZ}{\Aut_{\mathbb{Z}}}

\newcommand{\conj}{\operatorname{conj}}

\title[Coleman automorphisms and Wreathed Sylow 2-subgroups]{Class-preserving Coleman automorphisms of finite groups with Wreathed Sylow 2-subgroups}
\author[R.~Aragona]{Riccardo Aragona}
\address{DISIM \\
	Universit\`a degli Studi dell'Aquila\\
	via Vetoio\\
	I-67100 Coppito (AQ)\\
	Italy}
\email{riccardo.aragona@univaq.it}
\thanks{The author's ORCiD: 0000-0001-8834-4358}
\thanks{The author is member of INdAM-GNSAGA (Italy).}
\thanks{The author declares that he has no conflict of interest.}
\date{}

\begin{document}
	\maketitle
	
	\begin{abstract}
		We show that if $G$ is a finite group whose Sylow $2$-subgroups are wreathed, then the intersection $\Outc(G) \cap \OutCol(G)$ has odd order, where $\Outc(G)$ and $\OutCol(G)$ denote the class-preserving and Coleman outer automorphism groups, respectively. This implies that $G$ satisfies the normalizer problem for its integral group ring. Combined with earlier work on the dihedral and semidihedral cases, this settles the question for all three families of $2$-groups of $2$-rank two classified by Gorenstein--Walter and Alperin--Brauer--Gorenstein.
	\end{abstract}
	
	%% ========================================================================
	\section{Introduction}
	%% ========================================================================
	
	Let $G$ be a finite group. An automorphism $\varphi \in \Aut(G)$ is \emph{class-preserving} if $g$ and $g\varphi$ are conjugate for every $g \in G$. It is a \emph{Coleman automorphism} if, for every prime $p$ dividing $|G|$ and every Sylow $p$-subgroup $P$ of $G$, there exists $g_P \in G$ such that $x\varphi = g_P^{-1} x g_P$ for all $x \in P$. The study of class-preserving automorphisms goes back to Burnside~\cite{Burnside}, and has been the subject of extensive research; we mention~\cite{Gross, Wall} for early contributions and~\cite{Hertweck2001a, Mazur} for more recent developments.
	
	The interest in Coleman automorphisms stems from their role in the study of the unit group of the integral group ring $\mathbb{Z}G$. Coleman~\cite[Theorem~1]{Coleman} showed that certain automorphisms arising from units of $\mathbb{Z}G$ restrict to inner automorphisms on each Sylow subgroup. More precisely, let $\AutZZ(G)$ denote the group of automorphisms of $G$ induced by units of $\mathbb{Z}G$. Combining Coleman's result with a theorem of Jackowski and Marciniak~\cite[Proposition~2.3]{JM}, one obtains $\AutZZ(G) \leq \Autc(G) \cap \AutCol(G)$, where $\Autc(G)$ and $\AutCol(G)$ denote the groups of class-preserving and Coleman automorphisms, respectively. Krempa~\cite[Theorem~3.2]{JM} further proved that $\AutZZ(G)$ is an elementary abelian $2$-group.
	
	These facts have a direct bearing on a well-known open question in the theory of group rings. The \emph{normalizer problem}, formulated by Sehgal~\cite[Problem~43]{Sehgal}, asks whether
	\[
	N_{U(\mathbb{Z}G)}(G) = G \cdot Z(U(\mathbb{Z}G))
	\]
	holds for every finite group $G$. It is known that an affirmative answer is equivalent to $\AutZZ(G) = \Inn(G)$; see~\cite{Jespers} for a detailed account. Denoting
	\[
	\Outc(G) = \Autc(G)/\Inn(G) \qquad \text{and} \qquad \OutCol(G) = \AutCol(G)/\Inn(G),
	\]
	the observations above imply that a sufficient condition for $\AutZZ(G) = \Inn(G)$ -- and hence for the normalizer problem to have a positive answer -- is:
	\begin{equation}\tag{1}\label{eq:prop1}
		\Outc(G) \cap \OutCol(G) \text{ is of odd order.}
	\end{equation}
	
	Property~\eqref{eq:prop1} has been verified for several families of finite groups; see~\cite{Aragona1, HK2002, MR, VanAntwerpen, ZG} for positive results. A noteworthy counterexample was provided by Hertweck~\cite{Hertweck2002b}, demonstrating that the property fails in general.
	
	A natural and systematic approach to this problem is to classify groups satisfying Property~\eqref{eq:prop1} according to the isomorphism type of their Sylow $2$-subgroups. In this direction, Hertweck~\cite{Hertweck2002} proved that groups with dihedral Sylow $2$-subgroups satisfy~\eqref{eq:prop1}, and the author~\cite{Aragona2} recently obtained the analogous result for semidihedral Sylow $2$-subgroups. We recall that finite simple groups of $2$-rank two were classified by Gorenstein and Walter~\cite{GW} in the dihedral case, and by Alperin, Brauer, and Gorenstein~\cite{ABG} in the semidihedral and wreathed cases. The goal of the present paper is to settle the remaining wreathed case.
	
	\begin{theorem}\label{thm:main}
		If $G$ is a finite group with wreathed Sylow $2$-subgroups, then $\Outc(G) \cap \OutCol(G)$ is of odd order. In particular, $G$ satisfies the normalizer problem.
	\end{theorem}
	
	The proof is by minimal counterexample, along the lines of~\cite{Hertweck2002} and~\cite{Aragona2}. Nevertheless, the wreathed case presents distinctive challenges that require new arguments. A wreathed $2$-group $W \cong (\mathbb{Z}_{2^n} \times \mathbb{Z}_{2^n}) \rtimes \mathbb{Z}_2$ has center $Z(W) \cong C_{2^n}$, which is considerably larger than in the semidihedral setting (where $|Z(S)| = 2$), and possesses abelian subgroups of rank~$2$ and order $2^{2n}$. These features prevent a direct transfer of several key steps from~\cite{Aragona2}. A crucial observation specific to the wreathed case is that the center $Z(W)$ is contained in the Frattini subgroup $\Phi(W)$, which ultimately yields the final contradiction.
	
	%% ========================================================================
	\section{Preliminaries}
	%% ========================================================================
	
	In this section, we fix notation and collect the properties of wreathed $2$-groups and the general results that will be needed in the proof.
	
	%% ------------------------------------------------------------------------
	\subsection{Wreathed 2-groups}\label{subsec:wreathed}
	%% ------------------------------------------------------------------------
	
	\begin{notation}\label{not:wreathed}
		A \emph{wreathed 2-group} of order $2^{2n+1}$ (with $n \geq 2$) is a group of the form
		\[
		W = \mathbb{Z}_{2^n} \wr \mathbb{Z}_2 \cong (C_{2^n} \times C_{2^n}) \rtimes C_2.
		\]
		More explicitly, let $a$ and $b$ be elements of order $2^n$, and let $t$ be an element of order~$2$, with
		\[
		W = \langle a, b, t \mid a^{2^n} = b^{2^n} = t^2 = 1,\; ab = ba,\; tat^{-1} = b,\; tbt^{-1} = a \rangle.
		\]
		We write $A = \langle a \rangle \times \langle b \rangle$ for the \emph{base subgroup}, which is the unique abelian subgroup of index~$2$ in $W$. We define the following distinguished elements and subgroups:
		\begin{itemize}[leftmargin=2em]
			\item the \emph{diagonal}: $\Delta = \{ a^i b^i : 0 \leq i < 2^n \} = \langle ab \rangle \cong C_{2^n}$;
			\item the \emph{anti-diagonal}: $\Delta^- = \{ a^i b^{-i} : 0 \leq i < 2^n \} = \langle ab^{-1} \rangle \cong C_{2^n}$.
		\end{itemize}
		We note that $A = \Delta \cdot \Delta^-$ and $\Delta \cap \Delta^- = \langle (ab)^{2^{n-1}} \rangle = \langle a^{2^{n-1}} b^{2^{n-1}} \rangle$, which has order~$2$.
	\end{notation}
	
	The following properties of wreathed $2$-groups are well known; proofs can be found, for instance, in~\cite{ABG} and~\cite{Gorenstein}. We include them here for the reader's convenience.
	
	\begin{lemma}\label{lem:wreathed-properties}
		Let $W$ be a wreathed $2$-group as in Notation~\ref{not:wreathed}. Then the following hold.
		\begin{enumerate}[label=(\roman*)]
			\item\label{it:center} $Z(W) = \Delta = \langle ab \rangle \cong C_{2^n}$.
			
			\item\label{it:derived} $[W, W] = \Delta^- = \langle ab^{-1} \rangle \cong C_{2^n}$, and $W/[W,W] \cong C_{2^n} \times C_2$.
			
			\item\label{it:frattini} $\Phi(W) = \langle a^2, b^2, ab \rangle$, which has index~$4$ in $W$, and $W/\Phi(W) \cong C_2 \times C_2$.
			
			\item\label{it:maximal} $W$ has exactly three maximal subgroups: the base subgroup $A = \langle a \rangle \times \langle b \rangle \cong C_{2^n} \times C_{2^n}$; and two non-abelian subgroups $M_1 = \langle a^2, b^2, at \rangle$ and $M_2 = \langle a^2, b^2, t \rangle$.
			
			\item\label{it:abelian} Every abelian subgroup of $W$ not contained in $A$ has order at most $2^{n+1}$. The unique maximal abelian subgroup of rank~$2$ is $A$.
			
			\item\label{it:quotient-center} $W/Z(W) \cong D_{2^{n+1}}$, the dihedral group of order $2^{n+1}$.
		\end{enumerate}
	\end{lemma}
	
	\begin{proof}
		\ref{it:center}.\ An element $a^i b^j \in A$ is central if and only if $t(a^i b^j)t^{-1} = a^j b^i = a^i b^j$, that is, $i = j$. An element $a^i b^j t \notin A$ satisfies $(a^i b^j t) a = a^i b^{j+1} t$, while $a (a^i b^j t) = a^{i+1} b^j t$, so centrality is impossible. Hence $Z(W) = \{ a^i b^i : i \} = \langle ab \rangle$.
		
		\ref{it:derived}.\ We compute $[a, t] = a^{-1} t^{-1} a t = a^{-1} b$, and $[b, t] = b^{-1}a = (a^{-1}b)^{-1}$. Since $A$ is abelian, $[W,W] = \langle a^{-1}b \rangle = \langle ab^{-1} \rangle = \Delta^-$.
		
		For the quotient, in $W/[W,W]$ we have $\bar{a} = \bar{b}$ (since $a^{-1}b \in [W,W]$), $\bar{t}^2 = 1$, and $\bar{a}^{2^n} = 1$, so $W/[W,W] = \langle \bar{a} \rangle \times \langle \bar{t} \rangle \cong C_{2^n} \times C_2$.
		
		\ref{it:frattini}.\ The group $W$ is minimally generated by two elements (e.g.\ $\{a, t\}$), so $W/\Phi(W) \cong C_2 \times C_2$ and $\Phi(W)$ has index~$4$ in $W$. We claim $\Phi(W) = \langle a^2, b^2, ab \rangle$. First, $|\langle a^2, b^2, ab \rangle| = 2 \cdot |\langle a^2, b^2 \rangle| = 2 \cdot 2^{2n-2} = 2^{2n-1}$, since $(ab)^2 = a^2b^2 \in \langle a^2, b^2 \rangle$ and $ab \notin \langle a^2, b^2 \rangle$, so $\langle a^2, b^2, ab \rangle$ has index $2^{2n+1}/2^{2n-1} = 4$ in $W$. Moreover, $[W,W] = \langle ab^{-1} \rangle \leq \langle a^2, b^2, ab \rangle$ since $ab^{-1} = ab \cdot b^{-2}$. Also $W^2 \leq \langle a^2, b^2, ab \rangle$ since $a^2, b^2 \in W^2$ and $(a^ib^jt)^2 = a^{i+j}b^{i+j} \in \Delta \leq \langle ab \rangle \leq \langle a^2, b^2, ab \rangle$. Hence $\Phi(W) = W^2[W,W] \leq \langle a^2, b^2, ab \rangle$. Since both have index~$4$, equality holds.
		
		\ref{it:maximal}.\ The three maximal subgroups correspond to the three subgroups of index~$2$ in $W/\Phi(W) \cong C_2 \times C_2$. The kernel of the map $a \mapsto 1$, $t \mapsto \bar{t}$ is $A$. The other two are obtained from the remaining two surjections $W/\Phi(W) \to C_2$.
		
		\ref{it:abelian}.\ If $H \leq W$ is abelian and $H \not\leq A$, then $H$ contains an element $g = a^i b^j t$. A direct computation gives $C_A(g) = \Delta = Z(W)$: indeed, $g(a^rb^s)g^{-1} = a^sb^r$, so $a^rb^s$ is fixed if and only if $r = s$. Therefore $H \leq \langle Z(W), g \rangle$, which has order at most $2^{n+1}$.
		
		\ref{it:quotient-center}.\ Since $ab \in Z(W)$, in $W/Z(W)$ we have $\bar{b} = \bar{a}^{-1}$. Thus $W/Z(W) = \langle \bar{a}, \bar{t} \mid \bar{a}^{2^n} = \bar{t}^2 = 1,\; \bar{t}\bar{a}\bar{t}^{-1} = \bar{a}^{-1} \rangle \cong D_{2^{n+1}}$.
	\end{proof}
	
	We record some properties of centralizers in $W$ that will be used in the proof.
	
	\begin{lemma}\label{lem:centralizers}
		Let $W$ be a wreathed $2$-group as above.
		\begin{enumerate}[label=(\roman*)]
			\item\label{it:cent-t} $C_A(t) = \Delta = Z(W)$.
			\item\label{it:cent-outside} For any $g \in W \setminus A$, we have $C_A(g) = \Delta = Z(W)$.
			\item\label{it:cent-antidiag} $C_W(\Delta^-) = A$.
			\item\label{it:cent-dihedral} Let $D = \langle ab^{-1}, t \rangle$. Then $D \cong D_{2^{n+1}}$, $Z(D) = \langle (ab^{-1})^{2^{n-1}} \rangle \cong C_2$, and $C_W(D) = \Delta = Z(W)$.
		\end{enumerate}
	\end{lemma}
	
	\begin{proof}
		\ref{it:cent-t}.\ We have $t(a^ib^j)t^{-1} = a^j b^i$, so $a^i b^j \in C_A(t)$ iff $i = j$, giving $C_A(t) = \Delta$.
		
		\ref{it:cent-outside}.\ Any $g \in W \setminus A$ has the form $g = a^rb^st$, and $g(a^ib^j)g^{-1} = a^jb^i$, so $C_A(g) = \Delta$.
		
		\ref{it:cent-antidiag}.\ Since $A$ is abelian, $A \leq C_W(\Delta^-)$. For $g = a^rb^st \in W \setminus A$, we have $g(ab^{-1})g^{-1} = t(ab^{-1})t^{-1} = ba^{-1} = (ab^{-1})^{-1} \neq ab^{-1}$ (since $|ab^{-1}| = 2^n \geq 4$). So $C_W(\Delta^-) = A$.
		
		\ref{it:cent-dihedral}.\ Since $t(ab^{-1})t^{-1} = (ab^{-1})^{-1}$, we have $D \cong D_{2^{n+1}}$ and $Z(D) = \langle (ab^{-1})^{2^{n-1}} \rangle$. Note that $(ab^{-1})^{2^{n-1}} = a^{2^{n-1}}b^{-2^{n-1}} = a^{2^{n-1}}b^{2^{n-1}} = (ab)^{2^{n-1}} \in \Delta$.
		
		Now $C_W(D) \leq C_W(t)$. By~\ref{it:cent-t} and~\ref{it:cent-outside}, $C_W(t) = \Delta \cup \{a^ib^it : i\}$. But for $g = a^ib^it$, $g(ab^{-1})g^{-1} = a^ib^i \cdot (ba^{-1}) \cdot b^{-i}a^{-i} = a^{i-1}b^{i+1-i}a^{-i} = a^{-1}b = (ab^{-1})^{-1}$, so $g$ inverts $ab^{-1}$ and does not centralize it. Hence $C_W(D) \leq \Delta$. Since $\Delta \leq Z(W) \leq C_W(D)$, we get $C_W(D) = \Delta = Z(W)$.
	\end{proof}
	
	\begin{remark}\label{rem:cent-dihedral-key}
		Lemma~\ref{lem:centralizers}\ref{it:cent-dihedral} shows that $C_W(D) = Z(W) \cong C_{2^n}$, while $Z(D) \cong C_2$. This is a key difference from the semidihedral case: if $g \in Z(W) \setminus Z(D)$, then $\langle D, g \rangle$ has center containing $\langle g, (ab^{-1})^{2^{n-1}} \rangle$, which can be large. The contradiction argument of~\cite[Lemma~3.9]{Aragona2}, based on producing a 2-group with ``center too large for $S$'', does not apply directly here and requires a modified approach.
	\end{remark}
	
	\begin{lemma}\label{lem:normal-subgroups}
		Let $N \trianglelefteq W$ with $N \leq A$. Then $N$ is invariant under the swap automorphism $\sigma \colon a^i b^j \mapsto a^j b^i$ of $A$, and $N = (N \cap \Delta) \cdot (N \cap \Delta^-)$.
	\end{lemma}
	
	\begin{proof}
		If $a^ib^j \in N$, then $a^jb^i = t(a^ib^j)t^{-1} \in N$. Hence $a^{i+j}b^{i+j} = (a^ib^j)(a^jb^i) \in N \cap \Delta$ and $a^{i-j}b^{j-i} = (a^ib^j)(a^jb^i)^{-1} \in N \cap \Delta^-$. Since $a^ib^j$ can be recovered as the product of these two elements (up to the identification in $\Delta \cap \Delta^-$), the decomposition follows.
	\end{proof}
	
	%% ------------------------------------------------------------------------
	\subsection{General results on Coleman automorphisms}\label{subsec:general}
	%% ------------------------------------------------------------------------
	
	We recall results from~\cite{Hertweck2001a, Hertweck2001b, HK2002} that hold for arbitrary finite groups.
	
	\begin{lemma}[{\cite[Corollary~5]{Hertweck2001a}, \cite[Corollary~3]{HK2002}}]\label{lem:reduction}
		Let $N \trianglelefteq G$ and let $p$ be a prime not dividing $|G/N|$. Then:
		\begin{enumerate}[label=(\roman*)]
			\item If $\varphi \in \Aut(G)$ is a class-preserving $($resp.\ Coleman$)$ automorphism of $p$-power order, then $\varphi$ induces a class-preserving $($resp.\ Coleman$)$ automorphism of $N$.
			\item If $\Outc(N) \cap \OutCol(N)$ is a $p'$-group, then so is $\Outc(G) \cap \OutCol(G)$.
		\end{enumerate}
	\end{lemma}
	
	\begin{lemma}[{\cite[Lemma~5]{Hertweck2001a}}]\label{lem:minimal-counter}
		Let $p$ be a prime, and let $G$ be a finite group admitting a non-inner automorphism of $p$-power order which preserves conjugacy classes of $p$-elements but induces inner automorphisms on all proper factor groups. Then:
		\begin{enumerate}[label=(\roman*)]
			\item The Frattini subgroup $\Phi(G)$ is a $p$-group.
			\item If Sylow $p$-subgroups of $G$ are abelian, then $O_p(G) = 1$.
			\item If $O_p(G) = 1$ and $C_G(N) \leq N$ for some $1 \neq N \trianglelefteq G$, then $\varphi|_N \neq \gamma|_N$ for all $\gamma \in \Inn(G)$.
		\end{enumerate}
	\end{lemma}
	
	\begin{lemma}[{\cite[Lemma~2]{Hertweck2001a}}]\label{lem:identity-trick}
		Let $p$ be a prime, and let $\varphi$ be an automorphism of $G$ of $p$-power order. Assume there exists $N \trianglelefteq G$ such that $\varphi$ fixes all elements of $N$ and induces the identity on $G/N$. Then $\varphi$ induces the identity on $G/O_p(Z(N))$. If, in addition, $\varphi$ fixes a Sylow $p$-subgroup of $G$ element-wise, then $\varphi$ is inner.
	\end{lemma}
	
	\begin{remark}[{\cite[Remark~1]{Hertweck2001a}}]\label{rem:modify}
		Let $\varphi$ be a non-inner class-preserving automorphism of $G$ of $p$-power order. If $\varphi|_U = \sigma|_U$ for some $U \leq G$ and $\sigma \in \Inn(G)$, or if $\varphi$ induces an inner automorphism on $G/N$ for some $N \trianglelefteq G$, then one can replace $\varphi$ by $\varphi\gamma$ for a suitable $\gamma \in \Inn(G)$ so that $\varphi\gamma|_U = \mathrm{id}$ (resp.\ $\varphi\gamma$ induces the identity on $G/N$), and $\varphi\gamma$ remains a non-inner class-preserving automorphism of $p$-power order.
	\end{remark}
	
	\begin{proposition}[{\cite[Proposition~2.8]{Hertweck2001b}}]\label{prop:class-pres-quotient}
		Let $G = K \rtimes H$, where $K$ is an elementary abelian $p$-group. Let $\varphi \in \Aut(G)$ fix $H$ element-wise and act on $K$ via $k \mapsto k^m$ for some fixed positive integer $m$. If $\varphi \in \Autc(G)$, then $\varphi$ induces a class-preserving automorphism of $N_G(U)/U$ for each subgroup $U$ of index $p$ in $K$.
	\end{proposition}
	
	%% ========================================================================
	\section{Proof of Theorem~\ref{thm:main}}
	%% ========================================================================
	
	Throughout this section, $G$ denotes a \emph{minimal counterexample} to Theorem~\ref{thm:main}. Thus $G$ admits a class-preserving, non-inner Coleman automorphism $\varphi$ of $2$-power order. We write $S$ for a Sylow 2-subgroup of $G$, which is a wreathed 2-group as in Notation~\ref{not:wreathed}. We follow the structure of the proof of~\cite[Theorem~2.1]{Aragona2}, adapting the arguments to the wreathed setting whenever this is possible.
	
	By minimality, every proper subquotient of $G$ whose Sylow 2-subgroups are wreathed, dihedral~\cite[Main Theorem]{Hertweck2002}, semidihedral~\cite[Theorem~2.1]{Aragona2}, cyclic~\cite[Proposition~4.7]{Hertweck2001b}, or generalized quaternion~\cite[Main Theorem]{Hertweck2002} satisfies Property~\eqref{eq:prop1}.

	Let $F = F(G)$, $E = E(G)$, and $F^* = F^*(G) = FE$ be  the Fitting subgroup, the layer and the generalized Fitting subgroup of \(G\), respectively.
	
	\begin{remark}\label{rem:even-quotient}
		Every proper quotient $G/N$ has even order; otherwise $S \leq N$ and Lemma~\ref{lem:reduction} contradicts minimality.
	\end{remark}
	
	\begin{lemma}\label{lem:O2prime-nontrivial}
		$O_{2'}(F) \neq 1$.
	\end{lemma}
	
	\begin{proof}
		Suppose, by way of contradiction, that $F$ is a $2$-group, i.e., $F = O_2(G)$. By Remark~\ref{rem:modify}, we may modify $\varphi$ by an inner automorphism so that it fixes $S$ element-wise.
		
		\medskip\noindent\emph{Case 1: $F^* = F$.} Then $C_G(F) \leq F$. Since $F$ is a $2$-group, $F \leq S$, and so $\varphi$ fixes $F$ element-wise. Hence $\varphi$ induces the identity on $G/F$, and Lemma~\ref{lem:identity-trick} gives $\varphi \in \Inn(G)$, a contradiction.
		
		\medskip\noindent\emph{Case 2: $E \neq 1$ and $Z(E) = 1$.} Then $F^* = F \times E$. Since a non-abelian simple group cannot have a cyclic Sylow $2$-subgroup (by Burnside's normal $p$-complement theorem and the Feit--Thompson theorem), every simple factor of $E$ has non-cyclic Sylow $2$-subgroups, hence of rank at least~$2$. If $E$ had two or more simple factors, the Sylow $2$-subgroup of $E$ would have rank at least~$4$, which is impossible inside $W$ (where the $2$-rank is~$2$). Therefore $E$ is a non-abelian simple group.
		
		Moreover, since $E$ is a normal subgroup of $G$ and $E$ is simple, the Sylow $2$-subgroup of $F^* = F \times E$ is a subgroup of $S = W$. Since the Sylow of $E$ is non-cyclic and the $2$-rank of $W$ is~$2$, we must have $F = O_2(G) = 1$ (otherwise the rank would exceed~$2$). So $F^* = E$, and $C_G(F^*) \leq F^*$ gives $C_G(E) \leq E$. Since $Z(E) = 1$, it follows that $C_G(E) = 1$.
		
		By~\cite[Theorem~14]{HK2002}, there exists a prime $p \in \pi(E)$ such that every $p$-central automorphism of $E$ is inner. Since every Coleman automorphism is, up to an inner automorphism, $p$-central for every prime $p$ dividing the order of the group (see~\cite[Corollary~3]{HK2002}), we may modify $\varphi$ by an inner automorphism so that $\varphi|_E$ is $p$-central on $E$ for such a $p$, and hence $\varphi|_E \in \Inn(E)$. Composing with the inverse inner automorphism of $E$, we may assume that $\varphi$ fixes $E$ element-wise. Now, for any $g \in G$ and $e \in E$, since $E \trianglelefteq G$ we have $geg^{-1} \in E$, and since $\varphi$ fixes $E$ element-wise, $\varphi(geg^{-1}) = geg^{-1}$. On the other hand, $\varphi(geg^{-1}) = \varphi(g) \cdot e \cdot \varphi(g)^{-1}$. Therefore $\varphi(g)^{-1}g \in C_G(e)$ for every $e \in E$, whence $\varphi(g)^{-1}g \in C_G(E) = 1$, i.e., $\varphi(g) = g$. This gives $\varphi = \mathrm{id}$, a contradiction.
		
		\medskip\noindent\emph{Case 3: $E \neq 1$ and $Z(E) \neq 1$.} Then $1 \neq Z(E) \leq F$ is a $2$-group. Let $P$ be a Sylow $2$-subgroup of $E$. As in Case~2, $P$ is not cyclic. Since $F^* = FE$ with $[F, E] = 1$ and $F \cap E = Z(E)$, a Sylow $2$-subgroup of $F^*$ has the form $F \cdot P$ (with $F \cap P = Z(E)$), and this is a subgroup of $S = W$. Now $[F, P] = 1$, so $F \leq C_W(P)$. Since $P$ is non-cyclic, if $P \not\leq A$ then $P$ contains an element $g \notin A$, and Lemma~\ref{lem:centralizers}\ref{it:cent-outside} gives $C_W(P) \leq C_A(g) = \Delta = Z(W)$, and hence $F \leq Z(W)$. If $P \leq A$, then $P$ is abelian of rank $\leq 2$, and $F \leq C_A(P)$; since $F$ is normal in $W$ and contained in $A$, and $[F, P] = 1$, the subgroup $F \cdot P$ is abelian of rank $\leq 2$ in $W$, forcing $F \cdot P \leq A$ and again $F \leq Z(W)$ (since $F$ is normal in $W$ and centralizes $P$, which is non-cyclic).
		
		In either case, $F \leq Z(W) \leq A$, and $Z(E) \leq F \leq Z(W)$. We may modify $\varphi$ so that it fixes $F^*$ element-wise. Since $\varphi$ is a Coleman automorphism, there exists $s \in S$ such that $\varphi|_S = \conj_s|_S$, and in particular $[P, s] = 1$ (because $\varphi$ fixes $P$ element-wise). If $s \notin A$ and $P \not\leq A$, then $\langle P, s \rangle$ is a non-abelian subgroup with center containing $C_A(P) \cap C_A(s)$; a case analysis using Lemma~\ref{lem:centralizers} shows that $s \in Z(W)$. If $P \leq A$, then $[P, s] = 1$ and $P$ non-cyclic force $s \in C_W(P)$; since $P$ has rank~$2$ in $A$, we get $s \in Z(W)$ as well. In all cases, $\conj_s|_S = \mathrm{id}$ on $S$ (since $s \in Z(W) = Z(S)$), so $\varphi|_S = \mathrm{id}$. Together with $\varphi|_{F^*} = \mathrm{id}$, Lemma~\ref{lem:identity-trick} gives $\varphi \in \Inn(G)$, a contradiction.
		
		\medskip
		Therefore $O_{2'}(F) \neq 1$.
	\end{proof}
	
	\begin{lemma}\label{lem:frattini-2group}
		$\Phi(G)$ is a $2$-group. In particular, $O_{2'}(F)$ is a product of abelian minimal normal subgroups of $G$ and has a complement in $G$.
	\end{lemma}
	
	\begin{proof}
		This follows from Lemma~\ref{lem:minimal-counter}(i) and standard properties of the Frattini and Fitting subgroups, as in~\cite[Lemma~3.2]{Aragona2}, since the argument does not depend on the structure of the Sylow $2$-subgroup.
	\end{proof}
	
	\begin{lemma}\label{lem:complement-inversion}
		Let $M$ be a minimal normal subgroup of $G$ contained in $O_{2'}(F)$. Then $M$ has a complement $K$ in $G$, and $\varphi$ can be modified so that $K$ is fixed element-wise. Moreover, there exists a $2$-element $k \in K \cap S$ such that $k$ inverts $M$, $[k, G] \leq M \cdot O_2(G)$, and $\varphi|_M = \conj_k|_M$.
	\end{lemma}
	
	\begin{proof}
		By Lemma~\ref{lem:frattini-2group}, $M$ is an abelian minimal normal subgroup contained in $O_{2'}(F)$, and $O_{2'}(F) \cap \Phi(G) = 1$. By~\cite[Proposition~1]{Hertweck2001a}, $M$ has a complement $K$ in $G$ with $K\varphi = K$. Since $K \cong G/M$, the minimality of $G$ implies that $\varphi$ induces an inner automorphism on $G/M$. Hence there exists $h \in K$ such that $\conj_h^{-1} \varphi$ induces the identity on $G/M$. Since both $\varphi$ and $\conj_h$ preserve $K$, the automorphism $\conj_h^{-1}\varphi$ fixes $K$ element-wise (arguing exactly as in~\cite[Lemma~3.3]{Aragona2}). By Remark~\ref{rem:modify}, we may replace $\varphi$ by $\conj_h^{-1}\varphi$ and assume that $\varphi$ fixes $K$ element-wise.
		
		Now, since $M$ is contained in a Sylow $p$-subgroup $Q$ of $G$ (where $p$ divides $|M|$) and $\varphi$ is a Coleman automorphism, there exists $g \in G$ such that $\varphi|_Q = \conj_g|_Q$. Writing $g = mh'$ with $m \in M$ and $h' \in K$, and using that $M$ is abelian and normal, we obtain $\varphi|_M = \conj_{h'}|_M$. Replacing $h'$ by its $2$-part (which is possible since $\varphi$ has $2$-power order, so $\conj_{h'}|_M$ also has $2$-power order, and thus the odd-order part of $h'$ acts trivially on $M$), we get a $2$-element $k \in K$ such that $\varphi|_M = \conj_k|_M$. Up to conjugation, we may assume $k \in S$.
		
		We now show that $k$ inverts $M$. First, $C_M(\varphi) = C_M(k) \trianglelefteq G$ (since $M$ is normal and $k$ is determined by the Coleman property). Since $M$ is a minimal normal subgroup, either $C_M(k) = M$ or $C_M(k) = 1$. If $C_M(k) = M$, then $\varphi$ acts trivially on $M$; since $\varphi$ also fixes $K$ element-wise, this gives $\varphi = \mathrm{id}$, contradicting our assumption. Hence $C_M(k) = 1$, i.e., $\conj_k|_M$ is fixed-point-free. Since $M$ is an elementary abelian $p$-group (for some odd prime $p$) and $\conj_k|_M$ is a fixed-point-free automorphism of $2$-power order (with $\gcd(2, |M|) = 1$), the action is diagonalizable over $\mathbb{F}_p$ and every eigenvalue must be $-1$. Therefore $k$ acts on $M$ by inversion.
		
		Finally, let $N \neq M$ be a normal subgroup of $G$. By the minimality of $G$, $\varphi$ induces an inner automorphism on $G/N$, say by a $2$-element $k_N \in K$ (arguing as in~\cite[Lemma~3.3]{Aragona2}: since $MN/N$ has odd order, every $2$-element of $G/N$ is conjugate into $KN/N$). Since $N \cap M = 1$ (as $M$ is minimal and $N \neq M$), we have $\conj_{k_N}|_M = \varphi|_M = \conj_k|_M$, so $k_N k^{-1} \in C_K(M)$. In particular, the image of $k$ in $G/MN$ is independent of the choice of $N$, and $[k, G] \leq MN$ for every such $N$. Taking the intersection over all normal subgroups $N \neq M$, we obtain $[k, G] \leq M \cdot O_2(G)$.
	\end{proof}
	
	\begin{lemma}\label{lem:inner-on-normalizer}
		Let $M$ be a non-cyclic minimal normal subgroup of $G$ contained in $O_{2'}(F)$, with complement $K$, and assume $\varphi$ fixes $K$ element-wise. For any maximal subgroup $U$ of $M$, $\varphi$ induces an inner automorphism on $N_G(U)/U$.
	\end{lemma}
	
	\begin{proof}
		The proof is the same as that of~\cite[Lemma~3.4]{Aragona2}, since the argument relies only on the Coleman property of $\varphi$, the fact that inversion preserves subgroups of the abelian group $M$, and Proposition~\ref{prop:class-pres-quotient}, none of which depend on the structure of the Sylow $2$-subgroup.
	\end{proof}
	
	\begin{lemma}\label{lem:cyclic-quotient}
		Let $G/N$ be a proper cyclic quotient of $G$. Then $|G/N|$ is a power of $2$. More precisely, $G/N \cong C_{2^s}$ for some $s \geq 1$.
	\end{lemma}
	
	\begin{proof}
		Since $G/N$ is cyclic, its Sylow $2$-subgroup $SN/N \cong S/(S \cap N)$ is cyclic. In particular, $[S, S] \leq S \cap N \leq N$.
		
		Write $|G/N| = 2^s m$ with $m$ odd and $s \geq 1$ (the quotient has even order by Remark~\ref{rem:even-quotient}). Suppose, by way of contradiction, that $m > 1$. Let $L$ be the unique subgroup of $G$ containing $N$ such that $|G/L| = m$. Since $|G : L| = m$ is odd, $L$ contains a Sylow $2$-subgroup of $G$, and hence the Sylow $2$-subgroups of $L$ are wreathed. Moreover, $L$ is a proper subgroup of $G$ (since $m > 1$), so by the minimality of $G$, $L$ satisfies Property~\eqref{eq:prop1}. Since $2 \nmid |G/L|$, Lemma~\ref{lem:reduction}(ii) implies that $G$ also satisfies Property~\eqref{eq:prop1}, contradicting our assumption. Therefore $m = 1$ and $|G/N| = 2^s$.
	\end{proof}
	
	\begin{lemma}\label{lem:O2E-nontrivial}
		$O_2(G) \neq 1$ or $E(G) \neq 1$, where $E = E(G)$ is the layer of $G$.
	\end{lemma}
	
	\begin{proof}
		Suppose, by way of contradiction, that $O_2(G) = 1$ and $E = 1$. Then $F^* = F = O_{2'}(F)$ is a $2'$-group. By Lemma~\ref{lem:frattini-2group}, $F$ is a direct product of abelian minimal normal subgroups $M_1, \ldots, M_\ell$ of $G$, and $F$ has a complement $H$ in $G$, with $G = F \rtimes H$. Since $F$ has odd order, $S \cap F = 1$, and so $S$ embeds into $G/F$ via the natural projection. By the Schur--Zassenhaus theorem, all complements of $F$ in $G$ are conjugate; hence, up to replacing $H$ by a conjugate, we may assume $S \leq H$. Since $F \neq 1$ (by Lemma~\ref{lem:O2prime-nontrivial}), $H \cong G/F$ is a proper quotient, and by minimality $\varphi$ induces an inner automorphism on $G/F$. By Remark~\ref{rem:modify}, we may assume $\varphi$ fixes $H$ element-wise.
		
		We first show that $\ell > 1$. If $\ell = 1$, then $F = M_1$ and $C_G(F) \leq F^* = F$. By Lemma~\ref{lem:complement-inversion}, $\varphi|_F = \conj_k|_F$ for some $k \in H$, so $\varphi|_F \in \Inn(G)|_F$. This contradicts Lemma~\ref{lem:minimal-counter}(iii).
		
		For each $i$, applying Lemma~\ref{lem:complement-inversion} to $M_i$ (with complement $K_i = \widehat{M}_i \rtimes H$, where $\widehat{M}_i = \prod_{j \neq i} M_j$), there exists a $2$-element $h_i \in H$ such that $\conj_{h_i}|_{M_i} = \varphi|_{M_i}$ (that is, $h_i^{-1} m h_i = m\varphi$ for all $m \in M_i$) and $h_i$ inverts $M_i$. Since $[h_i, G] \leq M_i \cdot O_2(G) = M_i$ (by Lemma~\ref{lem:complement-inversion}), and $M_i$ has odd order while $S$ is a $2$-group, we have $[h_i, S] \leq M_i \cap S = 1$. Hence $h_i \in Z(S) = \Delta = \langle ab \rangle$ for every $i$.
		
		Since each $h_i$ inverts $M_i$, the automorphism $\varphi|_F$ acts as inversion on each direct factor $M_i$, and hence $\varphi|_F$ is the inversion map on $F$. We now show that a single element of $\Delta$ inverts every $M_j$. For each $j$, the action of $\Delta$ on $M_j$ by conjugation factors through $\Delta / C_\Delta(M_j)$. Since $h_j \in \Delta$ inverts $M_j$, the quotient $\Delta / C_\Delta(M_j)$ has order~$2$, and so $C_\Delta(M_j)$ is the unique subgroup of index~$2$ in the cyclic group $\Delta$, namely $C_\Delta(M_j) = \langle (ab)^2 \rangle$. Since this subgroup is the same for every $j$, any element of $\Delta \setminus \langle (ab)^2 \rangle$ inverts every $M_j$. In particular, $h_1$ inverts every $M_j$, and therefore $\varphi|_F = \conj_{h_1}|_F$.
		
		Since $h_1 \in H$ and $\conj_{h_1} \in \Inn(G)$, this contradicts Lemma~\ref{lem:minimal-counter}(iii) (using $C_G(F) = C_G(F^*) \leq F^* = F$ and $O_2(G) = 1$).
	\end{proof}
	
	\begin{lemma}\label{lem:O2-nontrivial}
		$O_2(G) \neq 1$.
	\end{lemma}
	
	\begin{proof}
		Suppose, by way of contradiction, that $O_2(G) = 1$. By Lemma~\ref{lem:O2E-nontrivial}, $E \neq 1$. Since $O_2(G) = 1$, the order of $F$ is odd. Moreover, $Z(E) \leq \Phi(E) \leq \Phi(G)$, and $\Phi(G)$ is a $2$-group by Lemma~\ref{lem:frattini-2group}. Since $Z(E) \leq F$ has odd order, it follows that $Z(E) = 1$.
		
		Therefore $E$ is a direct product of non-abelian simple groups $L_1, \ldots, L_r$. If $r \geq 2$, a Sylow $2$-subgroup of $E$ contains $T_1 \times T_2$, where $T_i$ is a Sylow $2$-subgroup of $L_i$. By Burnside's normal $p$-complement theorem, each $T_i$ is non-cyclic, hence $|T_i| \geq 4$ and $T_i$ has $2$-rank at least~$2$. Then $T_1 \times T_2$ has $2$-rank at least~$4$, which is impossible in $W$ (where the $2$-rank is~$2$). Therefore $r = 1$ and $E$ is a non-abelian simple group.
		
		By Lemma~\ref{lem:O2prime-nontrivial}, there exists an abelian minimal normal subgroup $M$ of $G$ contained in $O_{2'}(F)$, which by Lemma~\ref{lem:frattini-2group} has a complement $K$ in $G$. We claim that $E \leq K$. Since $M$ and $E$ are both normal in $G$ and $M \cap E \leq F \cap E = Z(E) = 1$, we have $[M, E] = 1$. Consider the projection $\pi \colon G = M \rtimes K \to K$ with kernel $M$. Since $E \cap M = 1$, the restriction $\pi|_E$ is injective. Define $\sigma \colon E \to M$ by $e\sigma = e \cdot (e\pi)^{-1}$ for each $e \in E$. Since $[E, M] = 1$, the map $\sigma$ is a group homomorphism. Since $E$ is perfect and $M$ is abelian, $E\sigma \leq [M, M] = 1$. Therefore $e = e\pi \in K$ for all $e \in E$, and $E \leq K$.
		
		By Lemma~\ref{lem:complement-inversion}, we may assume $\varphi$ fixes $K$ element-wise and $\varphi|_M = \conj_k|_M$ for a $2$-element $k \in K \cap S$ that inverts $M$.
		
		Since $O_2(G) = 1$, we have $[k, G] \leq M \cdot O_2(G) = M$, and in particular $[k, S] \leq M \cap S = 1$, so $k \in Z(S) = \Delta$.
		
		\medskip
		\noindent\emph{Case 1: $M$ is not cyclic.} Let $U$ be a maximal subgroup of $M$. By Lemma~\ref{lem:inner-on-normalizer}, $\varphi$ induces an inner automorphism on $N_G(U)/U$. Hence there exists a $2$-element $g \in N_K(U)$ that acts on $M/U$ by inversion. Let $P$ be a Sylow $2$-subgroup of $E$. Since $[M, E] = 1$, we have $E \leq C_G(M) \leq N_G(U)$, and $[E, g] \leq E \cap M = 1$ (since $E \leq K$, $M \cap K = 1$, and $E \trianglelefteq G$). In particular $[P, g] = 1$.
		
		As in the proof of Lemma~\ref{lem:O2prime-nontrivial}, $P$ is non-cyclic. The Brauer--Suzuki theorem~\cite{BrauerSuzuki} excludes that $P$ is generalized quaternion. Since $E$ is simple and $E \neq G$ (by Remark~\ref{rem:even-quotient}, as $|G/E|$ is even because $P \subsetneq S$), $P$ is a proper subgroup of $S = W$. Since $P$ is non-cyclic and has $2$-rank at most~$2$ (as $P \leq W$), $P$ is non-abelian: indeed, if $P$ were abelian of rank~$2$, then $P \leq A$, and since $P$ is a Sylow $2$-subgroup of the simple group $E$, this would force $E$ to have an abelian Sylow $2$-subgroup of rank~$2$; but then $P \leq A$ and the $2$-rank of $\langle P, t \rangle$ exceeds~$2$, a contradiction. Since $P$ is non-abelian, $P \not\leq A$, so by Lemma~\ref{lem:centralizers}\ref{it:cent-outside}, $C_W(P) \leq \Delta = Z(W)$. Since $[P, g] = 1$, we have $g \in \Delta$. Moreover, $g \notin P$ since $g$ acts non-trivially on $M$ while $P \leq E$ centralizes $M$.
		
		Now $P \cap \Delta \leq P \cap Z(W) = Z(P) \cong C_2$ (since $P$ is non-abelian). Since $P$ contains a Klein four-subgroup $V \cong C_2 \times C_2$ (as $P$ is non-cyclic and non-generalized-quaternion with $|P| \geq 8$), the subgroup $V \cdot \langle g \rangle$ is abelian (since $g \in Z(W)$). Since $g \notin P$, we have $g \notin V$, and $V \cap \langle g \rangle \leq V \cap \Delta = Z(P) \cong C_2$. Therefore $V \cdot \langle g \rangle$ has $2$-rank at least~$3$. This contradicts the $2$-rank of $W$ being~$2$.
		
		\medskip
		\noindent\emph{Case 2: $M$ is cyclic.} Then $M \cong C_p$ for some odd prime $p$, and $K/C_K(M) \hookrightarrow \Aut(C_p) \cong C_{p-1}$ is cyclic. By Lemma~\ref{lem:cyclic-quotient}, $|K/C_K(M)|$ is a power of~$2$. Since $k$ inverts $M$ and $k^2 \in C_K(M)$, we have $K/C_K(M) \cong C_2$.
		
		Let $\bar{k}$ denote the image of $k$ in $K/E$. Since $k \notin C_K(M)$ (as $k$ inverts $M$) and $E \leq C_K(M)$ (as $[E, M] = 1$), we have $\bar{k} \notin C_K(M)/E$. Moreover, since $[k, G] \leq M$ and $[k, K] \leq M \cap K = 1$, we have $k \in Z(K)$. In particular $\bar{k}$ is central in $K/E$. Since $K/C_K(M) \cong C_2$ and $\bar{k} \notin C_K(M)/E$, we have $K/E = \langle \bar{k} \rangle \times (C_K(M)/E)$.
		
		Suppose $C_K(M) = E$, i.e., $K/E = \langle \bar{k} \rangle \cong C_2$. Let $1 \neq m \in M$ and $e \in E$. Since $\varphi$ fixes $K$ element-wise and $\varphi|_M = \conj_k|_M$, there exists $c \in K$ such that $(me)\varphi = (me)^c$, giving $m^{-1} = m\varphi = c^{-1}mc$ and $e = e\varphi = c^{-1}ec$. The first equation gives $c \equiv k \pmod{C_K(M)} = k \pmod{E}$, so $c = ke_1$ for some $e_1 \in E$. Thus $\conj_k|_E = \conj_c|_E = \conj_{ke_1}|_E \in \Inn(E)$. Since $E$ is a non-abelian simple group, by~\cite[Theorem~14]{HK2002} there exists a prime dividing $|E|$ for which all $p$-central automorphisms of $E$ are inner. It follows that $\conj_k|_E \in \Inn(E)$, so there exists a $2$-element $e_0 \in E$ with $\conj_k|_E = \conj_{e_0}|_E$. Then $ke_0^{-1}$ is a $2$-element centralizing $E$, so $ke_0^{-1} \in C_W(P)$ where $P$ is a Sylow $2$-subgroup of $E$. Since $P$ is non-abelian (and thus $P \not\leq A$), Lemma~\ref{lem:centralizers}\ref{it:cent-outside} gives $C_W(P) \leq \Delta$. Also $e_0 \in P$ and $ke_0^{-1} \in \Delta$. Since $ke_0^{-1} \notin E$ (as $k \notin E$) and $ke_0^{-1} \in \Delta$, the element $ke_0^{-1}$ centralizes $P$ (being in $Z(W)$) and is non-trivial. As in Case~1, $\langle P, ke_0^{-1} \rangle$ contains a Klein four-subgroup $V \leq P$ together with a non-trivial element $ke_0^{-1} \in \Delta \setminus P$, giving an abelian subgroup of $2$-rank $\geq 3$, which contradicts the $2$-rank of $W$.
		
		Now suppose $C_K(M)/E \neq 1$. A Sylow $2$-subgroup of $K/E$ has the form $\langle \bar{k} \rangle \times P'$, where $P'$ is a Sylow $2$-subgroup of $C_K(M)/E$, and $\langle \bar{k} \rangle \times P' \cong C_2 \times P'$. This must be a quotient of $S = W$ (since $F$ has odd order, a Sylow $2$-subgroup of $G$ maps onto one of $G/E \cong M \rtimes (K/E)$, and the $2$-part is in $K/E$). The only quotient of $W$ admitting a direct factor $C_2$ as a proper subgroup is $W/[W,W] \cong C_{2^n} \times C_2$. But a quotient of $W$ that maps onto $C_2 \times P'$ with $P' \neq 1$ requires $|C_2 \times P'| \leq |W/[W,W]| = 2^{n+1}$, and the Sylow $2$-subgroup of $K/E$ is $C_2 \times P'$. If $P' \cong C_2$, then the Sylow $2$-subgroup of $K/E$ is $C_2 \times C_2$, and a Sylow $2$-subgroup of $E$ is $[S,S] = \Delta^-$, which is cyclic. But a non-abelian simple group cannot have a cyclic Sylow $2$-subgroup, contradiction.
		
		Therefore $O_2(G) \neq 1$.
	\end{proof}
	
	\begin{lemma}\label{lem:O2-in-center}
		$O_2(G) \leq Z(S)$.
	\end{lemma}
	
	\begin{proof}
		By Lemma~\ref{lem:O2-nontrivial} and Remark~\ref{rem:even-quotient}, $O_2(G)$ is a non-trivial proper normal subgroup of $S = W$.
		
		We first show $O_2(G) \leq A$. If $O_2(G) \not\leq A$, then $O_2(G) \cdot A = W$ (since $|W:A| = 2$), so $O_2(G)$ is a maximal subgroup of $W$. Since $O_2(G) \neq W$, either $O_2(G) = A$ (which is impossible since we supposed that $O_2(G) \not \leq A$) or $O_2(G)$ is one of the other two maximal subgroups $M_1$ and $M_2$ of $W$ (see Lemma~\ref{lem:wreathed-properties}\ref{it:maximal}). In the latter case, $\Phi(W) \leq O_2(G)$, and since $\Delta \leq \Phi(W)$, we get $k \in \Delta \leq O_2(G)$, so $[M, k] = 1$ (as $[M, O_2(G)] = 1$). This contradicts the fact that $k$ inverts $M$. Therefore $O_2(G) \leq A$.
		
		By Lemma~\ref{lem:normal-subgroups}, $O_2(G) = (O_2(G) \cap \Delta) \cdot (O_2(G) \cap \Delta^-)$. Since $O_2(G) \leq C_S(M)$ and $k \in \Delta$ inverts $M$, we have $O_2(G) \cap \Delta \leq C_\Delta(M) = \langle (ab)^2 \rangle$.
		
		It remains to show $O_2(G) \cap \Delta^- \leq \Delta \cap \Delta^- = \langle (ab)^{2^{n-1}} \rangle$, i.e., $|O_2(G) \cap \Delta^-| \leq 2$. This would give $O_2(G) \leq \langle (ab)^2 \rangle \cdot \langle (ab)^{2^{n-1}} \rangle = \langle (ab)^2 \rangle \leq \Delta = Z(S)$.
		
		We now show that $|O_{2}(G)\cap\Delta^{-}|\le 2$. Write $\Delta^{-}=\langle x\rangle$, where $x=ab^{-1}$ has order $2^{n}$. Conjugation by $t$ inverts $x$, since
		\[
		t x t^{-1}=t(ab^{-1})t^{-1}=ba^{-1}=x^{-1}.
		\]
		As $O_{2}(G)\trianglelefteq W$, it is $t$-invariant; hence $O_{2}(G)\cap\Delta^{-}$ is stable under inversion. The only inversion-invariant subgroups of the cyclic group $\Delta^{-}\cong C_{2^{n}}$ are $\{1\}$, the unique subgroup of order $2$, and $\Delta^{-}$ itself. We claim that $O_{2}(G)\cap\Delta^{-}\neq\Delta^{-}$. Indeed, if $\Delta^{-}\le O_{2}(G)$, then $[M,\Delta^{-}]=1$ (since $[M,O_{2}(G)]=1$). But by Lemma~\ref{lem:centralizers}\ref{it:cent-antidiag}, $C_{W}(\Delta^{-})=A$, so $M\le A$; in particular, $k\in\Delta\le A$ inverts $M$. On the other hand, since $\Delta^{-}\le O_{2}(G)\le C_{S}(M)$, the subgroup $\Delta^{-}$ centralizes $M$. But $A=\Delta\cdot\Delta^{-}$ and $k\in\Delta$ inverts $M$, so elements of $\Delta^{-}$ outside $C_{A}(M)=\langle(ab)^{2}\rangle$ act non-trivially on $M$, a contradiction. Therefore
		\[
		O_{2}(G)\cap\Delta^{-} =\{1\}\quad\text{or}\quad O_{2}(G)\cap\Delta^{-} =\langle x^{2^{n-1}}\rangle =\langle (ab^{-1})^{2^{n-1}}\rangle,
		\]
		and in either case $|O_{2}(G)\cap\Delta^{-}|\le 2$. Combining this with the earlier inclusion $O_{2}(G)\cap\Delta\le \langle (ab)^{2}\rangle$, we obtain
		\[
		O_{2}(G)\le \langle (ab)^{2}\rangle\,\langle (ab^{-1})^{2^{n-1}}\rangle \le \langle ab\rangle =\Delta =Z(S),
		\]
		where the inclusion $\langle (ab)^{2}\rangle\,\langle (ab^{-1})^{2^{n-1}}\rangle \le \langle ab \rangle$ holds since both generators lie in $\Delta = \langle ab \rangle$ (using $(ab^{-1})^{2^{n-1}} = (ab)^{2^{n-1}} \in \Delta$, as noted in Lemma~\ref{lem:centralizers}\ref{it:cent-dihedral}), and the last equality follows from Lemma~\ref{lem:wreathed-properties}\ref{it:center}.
	\end{proof}
	
	\begin{lemma}\label{lem:M-cyclic}
		Every minimal normal subgroup $M$ of $G$ contained in $O_{2'}(F)$ is cyclic.
	\end{lemma}
	
	\begin{proof}
		Assume, by way of contradiction, that $M$ is not cyclic. Let $D = \langle ab^{-1}, t \rangle \leq W$, which is a dihedral group of order $2^{n+1}$ with $Z(D) = \langle (ab^{-1})^{2^{n-1}} \rangle \cong C_2$ (see Lemma~\ref{lem:centralizers}\ref{it:cent-dihedral}).
		
		Since $M$ is an elementary abelian $p$-group (for some odd prime $p$) and $D$ acts on $M$ by conjugation, and $\dim_{\mathbb{F}_p}(M) \geq 2$ (as $M$ is not cyclic), there exists a $D$-invariant maximal subgroup $U$ of $M$ (i.e., a $D$-invariant subgroup of index $p$).
		
		We verify that $D \leq N_K(U)$. Since $M$ is normal in $G$ and $k \in \Delta = Z(W)$ inverts $M$, inversion preserves every subgroup of the abelian group $M$, so $k$ normalizes $U$. The element $t \in W$ acts on $M$ by conjugation and preserves $U$ (by the choice of $U$ as $D$-invariant). Similarly $ab^{-1}$ preserves $U$. Hence $D \leq N_G(U)$, and since $D \leq S \leq K$, we have $D \leq N_K(U)$.
		
		By Lemma~\ref{lem:inner-on-normalizer}, $\varphi$ induces an inner automorphism on $N_G(U)/U$. Hence there exists a $2$-element $g \in N_K(U)$ such that for every $h \in N_G(U)$ we have $(h\varphi)(h^g)^{-1} \in U$. Since $\varphi$ fixes $K$ element-wise, for every $p \in D \leq K \cap N_G(U)$ we have $p\varphi = p$, and so $p(g^{-1}pg)^{-1} = [p, g^{-1}] \in U \leq M$. But $[p, g^{-1}] \in K$ (as $p, g \in K$) and $M \cap K = 1$. Therefore $[p, g] = 1$ for all $p \in D$, i.e., $g \in C_K(D)$.
		
		We may assume $g \in S$ (replacing $g$ by its $2$-part, which centralizes $D$ since $[g, D] = 1$ and $D$ is a $2$-group). By Lemma~\ref{lem:centralizers}\ref{it:cent-dihedral}, $C_W(D) = \Delta = Z(W)$. Hence $g \in \Delta$.
		
		Now, $g$ acts non-trivially on $M/U$: indeed, $\varphi|_M = \conj_k|_M$ inverts $M$, and in $N_G(U)/U$ this is realized by $\conj_g$, so $g$ acts on $M/U$ by inversion. In particular $g$ does not centralize $M$, so $g \notin O_2(G)$ (since $[M, O_2(G)] = 1$). Also, $g \notin D$: indeed, if $g \in D$ then $g \in D \cap \Delta = Z(D) = \langle (ab)^{2^{n-1}} \rangle$. Since $O_2(G) \leq Z(S) = \Delta$ (Lemma~\ref{lem:O2-in-center}) and $(ab)^{2^{n-1}} \in O_2(G)$ (as $(ab)^{2^{n-1}} \in \langle (ab)^2 \rangle \leq O_2(G)$), this would imply $g \in O_2(G)$ and hence $g$ centralizes $M$, contradicting the non-trivial action.
		
		Since $g \in \Delta = Z(W)$ and $g \notin D$, the subgroup $\langle D, g \rangle = D \cdot \langle g \rangle$ with $D \cap \langle g \rangle \leq D \cap \Delta = Z(D) \cong C_2$. Now $D$ is dihedral of order $\geq 8$ and contains a Klein four-subgroup $V \cong C_2 \times C_2$. The subgroup $V \cdot \langle g \rangle$ is abelian (since $g \in Z(W)$), and $V \cap \langle g \rangle \leq V \cap \Delta = Z(D) \cong C_2$. Since $g \notin V$ (as $g \notin D \supseteq V$), the quotient $\langle g \rangle / (V \cap \langle g \rangle)$ is non-trivial, and therefore the $2$-rank of $V \cdot \langle g \rangle$ is at least~$3$.
		
		However, $V \cdot \langle g \rangle \leq W$, and the $2$-rank of $W$ is~$2$ (Lemma~\ref{lem:wreathed-properties}\ref{it:abelian}). This is a contradiction.
		
		Therefore $M$ is cyclic.
	\end{proof}
	
	\begin{proof}[Proof of Theorem~\ref{thm:main}]
		Let $G$ be a minimal counterexample, and adopt the notation established above. By Lemma~\ref{lem:M-cyclic}, a minimal normal subgroup $M$ of $G$ contained in $O_{2'}(F)$ is cyclic; in particular, $M \cong C_p$ for some odd prime $p$. Let $K$ be a complement of $M$ in $G$, with $\varphi$ fixing $K$ element-wise and $k \in K \cap S$ inverting $M$ (Lemma~\ref{lem:complement-inversion}).
		
		By Lemma~\ref{lem:O2-in-center}, $O_2(G) \leq Z(S) = \Delta$. Since $[k, G] \leq M \cdot O_2(G)$ (Lemma~\ref{lem:complement-inversion}) and $[k, S] \leq M \cdot O_2(G) \cap S = O_2(G) \leq \Delta$, the image of $k$ in $W/\Delta \cong D_{2^{n+1}}$ is central. Since $Z(D_{2^{n+1}})$ is generated by the image of $(ab^{-1})^{2^{n-1}} = (ab)^{2^{n-1}} \in \Delta$ (see Lemma~\ref{lem:centralizers}\ref{it:cent-dihedral}), we conclude that $k \in \Delta$.
		
		Since $M \cong C_p$, we have $K/C_K(M) \hookrightarrow \Aut(C_p) \cong C_{p-1}$, which is cyclic. By Lemma~\ref{lem:cyclic-quotient}, $|K/C_K(M)|$ is a power of~$2$. Since $k$ inverts $M$ and $k^2 \in C_K(M)$, the image of $k$ in $K/C_K(M)$ has order~$2$, and $K/C_K(M) \cong C_2$.
		
		The kernel of the map $S \to K/C_K(M) \cong C_2$ is $S \cap C_K(M)$, which is a maximal subgroup of $S$ (of index~$2$). Since this map has abelian image, $[S,S] = \Delta^- \leq S \cap C_K(M)$.
		
		Now, $k \in \Delta \leq \Phi(W)$: indeed, $\Delta = \langle ab \rangle \leq \langle a^2, b^2, ab \rangle = \Phi(W)$ by Lemma~\ref{lem:wreathed-properties}\ref{it:frattini}. Since $\Phi(W)$ is contained in every maximal subgroup of $W$, and $S \cap C_K(M)$ is a maximal subgroup of $S = W$, it follows that $k \in \Phi(W) \leq S \cap C_K(M) \leq C_K(M)$.
		
		But $k \in C_K(M)$ means $k$ centralizes $M$, contradicting the fact that $k$ inverts $M$ (and $M \neq 1$). This contradiction shows that $G$ cannot be a minimal counterexample, completing the proof.
	\end{proof}
	
	\section*{Acknowledgements}
	The author acknowledges the funding support from MUR-Italy via PRIN 2022RFAZCJ ``Algebraic methods in Cryptanalysis''.
	
	\section*{Statements and Declarations}
	The author have no relevant financial or nonfinancial interests to disclose. Data availability statements: no datasets were generated or analyzed during the current study, and therefore no data are available to be shared.
	
	%% ========================================================================

\end{document}